\documentclass[a4paper]{cas-sc}
\usepackage[authoryear,longnamesfirst]{natbib}

\usepackage{mathtools,cleveref,times,ulem,amsthm} 
\newcommand{\ambSp}[1]{\mathcal{F}_{#1}} 
\newcommand{\xx}{\textbf{x}}    
\newcommand{\ww}{\mathbf{w}}    
\newcommand{\dd}{\mathbf{d}}    
\newcommand{\F}{\mathbb{F}}

\newcommand{\C}{\mathbb{C}}     
\renewcommand{\P}{\mathbb{P}}

\newcommand{\qbinom}[2]{\genfrac{[}{]}{0pt}{}{#1}{#2}_q} 

\newtheorem{theorem}{Theorem}[section]
\newtheorem{corollary}{Corollary}[theorem]
\newtheorem{lemma}[theorem]{Lemma}
\newtheorem{definition}{Definition}[section]
\newtheorem{proposition}[theorem]{Proposition}
\newtheorem{example}[theorem]{Example}
\newtheorem{remark}[theorem]{Remark}
\newtheorem{problem}[theorem]{Problem}
\newtheorem{conjecture}[theorem]{Conjecture}

\begin{document}

\let\WriteBookmarks\relax
\def\floatpagepagefraction{1}
\def\textpagefraction{.001}

\shorttitle{Shallow NN expressivity over finite fields}  
\shortauthors{M. Zubkov et~al.}  

\title [mode = title]{Expressivity of Shallow Neural Networks Over Finite Fields}  
                
\tnotemark[1]

\tnotetext[1]{This work was supported by a Canada CIFAR AI Chair award and an NSERC Discovery Grant (DGECR-2020-00338) held by Elina Robeva. CW and SY were additionally supported by a NSERC USRA grant and PM and YC by a WLIURA USRA grant.}
   
\author[1]{Maksym Zubkov}
\cormark[1]
\cortext[cor1]{Corresponding author}
\ead{mzubkov@berkeley.edu}

\author[1]{Carol Wu}
\ead{cwu88@student.ubc.ca}

\author[1]{Shiwei Yang}
\ead{mich207@student.ubc.ca}

\author[1]{Param Mody}
\ead{parammody@gmail.com}

\author[1]{Yifei Chen}
\ead{cyf1231@student.ubc.ca}

\affiliation[1]{organization={Department of Mathematics, University of British Columbia},
            addressline={1984 Mathematics Rd}, 
            city={Vancouver},
            postcode={V6T 1Z2}, 
            state={BC},
            country={Canada}}

\begin{abstract}
    We study the expressivity of shallow polynomial neural networks (PNNs) with monomial activation functions over finite fields. For a given architecture, we define a neuromanifold as the image of the map from all possible network weights into the product of polynomial rings. We quantify the expressivity by the cardinality of the neuromanifold, and derive a natural lower and upper bound. This leads to counting rational points over finite fields, a problem closely linked to the Weil conjectures. Finally, we present an architecture that exhibits a striking difference in the neuromanifolds when considered over a characteristic zero versus a finite‐characteristic field, illustrating the critical role of field characteristic in the notion of expressivity. 
\end{abstract}

\begin{highlights}
\item Characterized shallow neural network expressivity over finite fields via a combinatorial counting framework.
\item Identified key structural differences between finite-field and complex-valued neural networks.
\item Computed exact cardinalities for families of representable function spaces.
\item Established general upper and lower bounds on expressivity for these architectures.
\end{highlights}


\maketitle

\section{Introduction}
Neural networks have excelled in wide-ranging fields from computer vision to natural language processing in recent years. The capacity of a neural network architecture to represent a class of functions is referred to as its \textit{expressivity}~\citep{GUHRING20}. Neural network dynamics and training are typically studied using statistical and probabilistic methods.
    
    However, algebro-geometric approaches to neural networks gained increased attention in recent years and have been applied to diverse architectures. The direction of \textit{neuroalgebraic geometry}~\citep{MARCHETTI2025} investigates networks through algebraic geometry, studying the properties of the functional space of their outputs, known as the \textit{neuromanifold}~\citep{TRAGER20}. 
    
    In this work, we study neuromanifolds over finite fields. This is motivated by the practical advantages of weight quantization, which reduces both energy consumption and storage requirements~\citep{YUAN23}. Moreover, counting rational points over finite fields connects to the Grothendieck trace formula and Weil conjectures ~\citep{DELIGNE1980}, which allow one to compute cohomology groups over finite fields ~\citep{BRYAN26} and to learn motivic structure of the space over complex numbers ~\citep{BAEZ25}.

    In particular, we consider shallow neural networks over finite fields with monomial activation functions $\sigma(x) = x^r$, where the network output is a $k$-tuple of homogeneous polynomials of degree $r$. We analyze the neural network architecture via the associated \textit{parameter map}, defined as the map from the parameter space to the $k$-fold Cartesian product of a polynomial ring over a finite field.
    
    Since the space of polynomials of bounded degree over a finite field is itself finite, we can quantify the expressivity of a network by computing the cardinality of its neuromanifold. If the neuromanifold coincides with the entire ambient space, we say it is \textit{filling}~\citep{KILEEL19}. The behavior of neuromanifolds over finite fields can differ drastically from their real or complex counterparts, even for small architectures. For instance, for the architecture $\dd = (2,2,2)$ with $r=2$, we find the neuromanifold over finite fields has cardinality approximately half that of the ambient space, whereas over the complex numbers, the neuromanifold is Zariski open.

\subsection{Related Work}
    The geometric structure of neural network function spaces was first investigated in~\citet{AMARI94}, where the term \textit{neuromanifold} was introduced to denote the function space associated with a given architecture. For a comprehensive survey of neural network expressivity, see~\citet{GUHRING20}.

    The expressivity of polynomial neural networks (PNNs) was first analyzed in~\citet{KILEEL19} through the notions of filling and thick architectures. The geometry of neuromanifolds and neurovarieties was subsequently developed in~\citet{KUBJAS24}, while the expressive power of PNNs was further investigated in~\citet{FINKEL25}. Questions of identifiability were addressed in~\citet{USEVICH25}, and a detailed study of singularities was presented in~\citet{SHAHVERDI25-1}. In practice, PNNs have been shown to achieve performance comparable to networks with non-polynomial activation functions~\citep{YAVARTANOO21, CHRYSOS22}. 
    
    To improve computational efficiency and reduce storage and energy consumption, neural networks with low-precision weights have attracted significant attention~\citep{COURBARIAUX14}. Related directions include networks with rational coefficients~\citep{AVERKOV25} and integer coefficients~\citep{WU18}. The expressive capabilities of neural networks with limited-range integer weights were analyzed in~\citet{DRAGHICI02}. Neural networks that operate exclusively with integers have been studied in~\citet{WANG22}.  
    
    Significant attention has been devoted to binary and ternary neural networks, where weights are restricted to $\{0,1\}$ (1-bit) and $\{-1,0,1\}$, respectively. The training of binarized neural networks with 1-bit weights and activations was first introduced in~\citet{HUBARA16}. Subsequently, ternary weight networks, which improve efficiency over binary ones, were proposed in~\citet{LIU23}.

    An extensive survey of neural network quantization was provided in~\citet{GHOLAMI22}. A detailed review of binary neural networks, covering activation and weights, with emphasis on hardware training advantages, was presented in~\citet{YUAN23}.

    Recent developments in large language model (LLM) architectures have shown that weight quantization can be applied with minimal loss in performance, while yielding substantial gains in energy efficiency and storage requirements~\citep{MA24}. A comprehensive study of quantization techniques for LLMs was presented in~\citet{RENREN24}.

It is going to be useful to mention the formulae that we will use across the paper below: 
\begin{itemize}
    \item A $q$-binomial coefficient is equal to
    \begin{equation}
    \label{eq:q-binomial-coefficient}
        \qbinom{a}{b} = \frac{(1-q^a) (1-q^{a-1}) \cdots (1-q^{a-b+1})}{(1-q) (1-q^2) \cdots (1-q^b)}
    \end{equation}
    \item In~\citep{CONRAD1995}, if $r =\sum_{t=0}^{s_r} r_t p^{\,t}$ is the base-$p$ expansion of the activation degree, then 
    \begin{equation}
    \label{eq: upper bound exponent}
        \gamma_{n,p}(r)=\prod_{t=0}^{s_r}\binom{r_t+n-1}{\,r_t}.
    \end{equation}
    \item According to~\citet[Section~1.7]{MORRISON06}, the number of $k\times n$ matrices of rank $m$ over a finite field $\F$ with $q$ elements is
    \begin{equation}
    \label{eq: k x n matrices of rank m}
        N_{q}(n,k,m)=\qbinom{k}{m}(q^n-1)(q^n-q)\dots(q^n-q^{m-1}).
    \end{equation}
    \item According to~\citet[Section~2]{WILLIAMS69} the number of symmetric $n \times n$ matrices of rank $m$ over a finite field $\F$ with $q$ elements is
    \begin{equation}
    \label{eq: n x n symmetric matrices of rank m}
        N_q(n,m) = \prod_{i = 1}^{\lfloor{\frac{m}{2}}\rfloor}\frac{q^{2i}}{q^{2i} - 1} \cdot \prod_{i = 0}^{m - 1} (q^{n - i} - 1).
    \end{equation}
\end{itemize}

\subsection{Major Contributions}
    Our main contributions are summarized as follows:
    \begin{enumerate}
        \item We formulate the expressivity of PNNs over finite fields in terms of the cardinality of their neuromanifold $\mathcal{P}_{\dd,r}$.
        \item We compare shallow networks over finite fields with their complex counterparts, highlighting the key differences between them.
        \item We compute the cardinalities of neuromanifolds for various architectures and establish a general upper and lower bound.
    \end{enumerate}
    A summary of the computed architectures is presented in Table 1 in Section~\ref{subsec:paper-outline}.

\subsection{Paper outline}
\label{subsec:paper-outline}
    In Section~\ref{sec:preliminary}, we introduce the notion of a shallow neural network over finite fields, its ambient space, and establish a general upper bound on the cardinality of neuromanifolds over prime fields. We then present results for several shallow neural network architectures, including:
    \begin{itemize}
        \item Section~\ref{sec:shallow-r=2-p!=2}: architectures $\dd = (n,m,k)$ with square activation $\sigma(x) = x^2$ over a finite field $\F_q$ of characteristic $\neq 2$, and a lower bound on for architectures $\dd = (n,n,k)$ with arbitrary monomial activation $\sigma(x) = x^r$,
        \item Section~\ref{sec:shallow-m=1-m=2}: architectures $\dd = (n,1,k)$ and $\dd = (n,2,k)$ with arbitrary monomial activation $\sigma(x) = x^r$ over a finite field $\F_q$ of arbitrary characteristic,
        \item Section~\ref{sec:shallow-p|r}: architectures $\dd = (n,m,k)$ with activation degree $r$ divisible by the characteristic $p$ of the field $\F_q$.
    \end{itemize}
    Finally, Section~\ref{sec:conclusion} concludes with remarks and a summary of the results.

\begin{center}
\begin{table}[h!]
\label{architecture table}
\caption{Main Results} \vspace{1em}
    \renewcommand{\arraystretch}{1.7}
	\begin{tabular} {c|c|c|c|c} 

    \hline
        $\dd$ & $r$ & $p$ & $\left|\mathcal{P}_{\dd,r}\right|$ & Reference\\
    \hline 
        $(n,m,1)$ & 2 & $p \neq 2$ &
        \parbox[c]{0.45\textwidth}{
            \[
            \min\left(\sum_{j=0}^m N_q(n,j),\,q^{\binom{n+r-1}{r}}\right)
            \]
        }
        
        & \ref{lem: r=2 p!=2 k=1 m<=n} \\ 
        \hline 
        $(n,1,k)$ & any & 
        any & \parbox[c]{0.45\textwidth}{
        \begin{gather*}
        \frac{(q^n-1)(q^k-1)}{(q-1)} + 1
        \end{gather*} } & \ref{lem: n_1_k} \\ \hline 
        $(n,2,k)$ & 
        $\neq p^i$ & 
        any & \parbox[c]{0.45\textwidth}{
        \begin{gather*}
            (q-1)\left|\overline{\mathcal{P}_{(n,2,1), r}}\right| \left|\mathbb{P}^{k-1}\right| \\
            + \binom{\left|\overline{\mathcal{P}_{(n,1,1), r}}\right|}{2}N_{q}(2,k,2) + 1
        \end{gather*} 
        } & \ref{lem: dd_n_2_k}\\ \hline 
        $(n,2,k)$ & $p^i$ & any & 
        \parbox[c]{0.45\textwidth}{
        \begin{gather*} (q-1)\left | \overline{\mathcal{P}_{(n,2,1), r}} \right | \left|\mathbb{P}^{k-1}\right| \\
        + \frac{2}{q(q+1)} \binom{\lvert  \overline{\mathcal{P}_{(n,1,1), r}}\rvert }{2} N_{q}(2,k,2) + 1 \end{gather*} } & \ref{lem: dd_n_2_k}
        \\  \hline 
    $(n,m,k)$ & $ip$ & any & $\left|\mathcal{P}_{\dd,i}\right|$ & \ref{prop:architecture-M_dd_ip_eq_M_dd_i} \\ \hline
    $(n,m,k)$ & $p^i$ & any  & $ \min\left(\sum_{j=1}^m N_q(n,k,j) + 1,\, q^{kn}\right)$ & \ref{prop: r = p^i} 
    \\ \hline
Upperbound:
    \\ \hline
        $(n,m,k)$ & $\text{any}$ & $\text{any}$ & $ |\mathcal{P}_{(n,m,k),r}| \le q^{\,\gamma_{n,p}(r) \cdot k}$ & \ref{lem:global_count_k=1}  
\\ \hline
\end{tabular}
\end{table}
\end{center}

\section{Preliminary}
\label{sec:preliminary}
    Let $\F$ be a field, and fix a triple of natural numbers $\dd=(n,m,k)$. Let $\sigma:\F\to\F$ be the \textit{monomial activation function} with the activation degree $r\in\mathbb{N}$, that is, $\sigma(x)=x^r$. A \textit{shallow neural network} $f_{\ww}:\F^n\to\F^k$ with architecture $\ww=(\dd,\sigma)$ is the composition
    \[
        \F^n\xrightarrow{W_1}\F^m\xrightarrow{\sigma_1}\F^m\xrightarrow{W_2}\F^k,
    \]
    where $W_1\in \F^{m\times n}, W_2\in \F^{k\times m}$ are matrices, and $\sigma_1:\F^m\to\F^m$ is the coordinate-wise application of the activation function $\sigma$. Equivalently,
    \[
        f_{\ww}(\xx)=W_2(\sigma_1(W_1\xx)),\quad \xx=
        \begin{bmatrix}
            x_1 & \dots & x_n
        \end{bmatrix}^T.
    \]
    For the remainder of this paper, we denote the entries of $W_1$ and $W_2$ by $a_{ij}$ and $b_{ij}$, respectively. We denote the $i$-th component of $f_{\ww}$ by $f_{i,\ww}$ and observe that it is a homogeneous polynomial of degree $r$ over $n$ variables of the form
    \begin{equation}
    \label{eq:general-nn-output}
        f_{i,\ww}(\xx)=\sum_{s=1}^{m}b_{is}(a_{s1}x_1+\dots+a_{sn}x_n)^r.
    \end{equation}
    For the rest of the paper, $\F$ indicates any field, $\F_q$ is a finite field, and $\F_p$ is a prime field.

    \noindent
    \textbf{About the ambient space.} Let $\ambSp{}(\F^{n},\F^{k})$ be the space of all functions from $\F^{n}$ to $\F^{k}$. The \textit{expressive power} of a neural network refers to the extent to which the neural network can approximate functions within the ambient space $\mathcal F(\F^{n}, \F^{k})$~\citep{GUHRING20}. However, in this paper we adopt an algebraic notion of expressivity following~\citet{KILEEL19}. In other words, we study the set of functions represented exactly by a given architecture.

    For a finite field $\F_q$, the evaluation map
    \begin{equation} \label{eq: eval_map}
        \begin{aligned}
            eval: (\F_q[x_1,\dots,x_n])^k&\longrightarrow \ambSp{}(\F_q^{n},\F_q^{k})\\
            (f_1,\dots,f_k)&\longmapsto [a\mapsto (f_1(a),\dots,f_k(a))]
        \end{aligned}
    \end{equation}
    is not injective. For example, if $q$ is prime and $n=k=1$, the zero polynomial and the polynomial $x^p-x$ both map to the zero function in $\ambSp{}(\F_p,\F_p)$ as $x^p\equiv x \bmod p$ by Fermat's little theorem.

    A space in which a neural network $f_{\ww}$ resides is known as an \textit{ambient space}. In this work, as a first step, we propose studying the \textit{algebraic} expressivity by considering neural networks with polynomial activation functions as elements of a polynomial ring. In other words, we take the ambient space to be a Cartesian product of polynomial rings. This suggests the following definition.
    \begin{definition}
        Polynomials $g_1,\dots,g_k \in \F_q[x_1,\dots,x_n]$ can be \textit{expressed} by a neural network $f_{\ww}$ with architecture $(\dd,\sigma)$ if there exist weights $\ww\in\F_q^{m(n+k))}$
        such that $f_{i,\ww}(\xx)$ and $g_{i}(\xx)$ are the same element in the polynomial ring $\F_q[x_1,\dots,x_n]$ for all $i$.
    \end{definition}
    \begin{remark}
        Over finite fields, the equality of two elements in $\mathbb{F}_q[x]$ is strictly stronger than equality of them as functions $\mathbb{F}_q^n \to \mathbb{F}_q$; we intentionally adopt the former to obtain an algebraic parametrization.
    \end{remark}
    
    Let $S^r(\F^n)$ be the space of homogeneous polynomials of degree $r$ over $n$ variables with coefficients in $\F$. Observe that $f_{i,\ww}$ is a homogeneous polynomial according to~\Cref{eq:general-nn-output}. Under our assumption that the neural network has a monomial activation function $x^r$ without bias, we can shrink our ambient space for this network architecture to the product of $S^r(\F^n)$.

    \begin{remark}
         For fields $\F$ with $char(\F) \neq 2$, the space of homogeneous polynomials of degree $r$ is naturally isomorphic to the space of symmetric tensors of order $r$.
    \end{remark}

    We define the \textit{parameter map}
    \begin{equation}
        \begin{aligned}
            \Psi_{\dd, r}: \F^{m(n+k)} & \rightarrow (S^r(\F^n))^k \\
            \ww &\mapsto f_{\ww}\coloneq(f_{1,\ww},\dots,f_{k,\ww})
        \end{aligned}
    \end{equation}
    
    \begin{remark}
        We can identify the space $S^r(\F^n)$ with the vector space $\F^{{\binom{n+r-1} {r}}}$ where we take the coefficients of a homogeneous polynomial in the standard monomial basis.
    \end{remark}
    \begin{definition}
        The \textit{neuromanifold}, denoted by $\mathcal{P}_{\dd,r}(\F)$, is the image of the parameter map $\Psi_{\dd,r}$.
        When working over a finite field $\F_q$, if unambiguous, we simply write $\mathcal{P}_{\dd,r}$.
    \end{definition}

    In the setting of finite fields, the set $\mathcal{P}_{\dd,r}$ is finite and therefore already Zariski closed in the ambient affine space. However, according to~\cite{GORTZ10}, we can change the coefficients from a finite field $\F_q$ to its algebraic closure $\overline{\F}_q$. Then the image of the parameter map over $\overline{\F}_q$ is no longer just a finite set of $\F_q$-points, and its Zariski closure will become a genuine algebraic variety. In this paper, we take an arithmetic viewpoint: we study algebraic expressivity via point counts of $\mathcal{P}_{\dd,r}(\F_q)$, and we leave the geometric discussion above to future work.

    A central problem in this context is to quantify how large $\mathcal{P}_{\dd,r}$ is inside the ambient space $(S^r(\F^n))^k$~\citep{KILEEL19}. In our finite-field setting, we measure this via the following notion of arithmetic capacity.
    
    \begin{definition}\label{def: expressive_capacity}
        The \textit{arithmetic expressive capacity} of the neural network with architecture $\dd=(n,m,k)$
        and monomial activation degree $r$ over $\F_q$ is the ratio
        \[
            \delta_{\dd,r}(q)
            \;:=\;
            \frac{|\mathcal{P}_{\dd,r}|}{\bigl|\left(S^r(\F^{\,n}_q)\right)^k\bigr|}.
        \]
    \end{definition}

    \begin{example}
        Let $d = (2,2,1), r = 2,$ and $char(\F_q)\neq2$. Then the output of the neural network is equal to
        \[
            f_{\ww}(\xx) = W_2(\sigma_1(W_1\xx)) = \begin{bmatrix}
                b_1 & b_2
            \end{bmatrix}\sigma_1\left(\begin{bmatrix}
                a_{11} & a_{12}\\
                a_{21} & a_{22}
            \end{bmatrix}\begin{bmatrix}
                x_1\\
                x_2
            \end{bmatrix}\right)
        \]
        A direct computation gives
        \[
            f_{\ww}(\xx) = (b_1 a_{11}^2 + b_2 a_{21}^2)x_1^2 + (2b_1 a_{11} a_{12} + 2b_2 a_{21} a_{22})x_1 x_2 + (b_1 a_{12}^2 + b_2 a_{22}^2)x_2^2
        \]
        Therefore, the parameter map $\Psi_{(2,2,1),2} : \F^6 \rightarrow S^2(\F^2)\cong\F^3$ is given by
        \[
            (a_{11}, a_{12}, a_{21}, a_{22}, b_1, b_2) \mapsto 
            \left(
                b_1 a_{11}^2 + b_2 a_{21}^2,\ 
                2b_1 a_{11} a_{12} + 2b_2 a_{21} a_{22},\ 
                b_1 a_{12}^2 + b_2 a_{22}^2
            \right).
        \]
    \end{example}

\noindent
\textbf{About the upper bounds for neuromanifolds.} To obtain a natural upper bound for neuromanifolds over a finite field $\F_q$ with characteristic $p$, we use the multinomial congruence theorem~\citep{CONRAD1995}. The key observation is that, in the multinomial expansion
\[
    l_i^r=(a_{i1}x_1+\cdots+a_{in}x_n)^r,
\]
some multinomial coefficients $\binom{r}{I}$ (and hence the corresponding coefficients $c_I$ of the monomials $x^I$) vanish modulo $p$. Consequently, in the expansion of
\[
    b_{j1}l_1^r+\cdots+b_{jm}l_m^r,
\]
the coefficients of those monomials $x^I$ are identically zero over $\F_q$ (i.e., they vanish for all choices of parameters). This observation leads to the following result.   
\begin{proposition}
\label{lem:global_count_k=1}
    If $\dd=(n,m,k)$ and $\sigma(x) = x^r$ with $r=\sum_{t=0}^{s_r} r_t p^{\,t}$ over a finite field $\F_q$ with $char(\F_q) = p$ and $\gamma_{n,p}(r)$ is as in~\Cref{eq: upper bound exponent}, the cardinality of the neuromanifold satisfies the following upper bound:
    \[
        |\mathcal P_{(n,m,k),r}(\F_q)|
        \;\le\;
        \left(q^{\,\gamma_{n,p}(r)}\right)^k.
    \]
\end{proposition}
\begin{proof}
    Consider the base-$p$ expansion of $r$ as $\sum_{t=0}^{s_r} r_t\,p^{\,t}$ with digits $0\le r_t\le p-1$. For each exponent $n$-tuple $\alpha=(\alpha_1,\dots ,\alpha_n) \in \mathbb{Z}_{\geq 0}^n$ with $\sum_{i=1}^n \alpha_i=r$, let $\alpha_{i,t}\in\{0,\dots ,p-1\}$ be its
    $t$-th base-$p$ digit so that $\alpha_i=\sum_t\alpha_{i,t} p^{\,t}$. 
    
    By Theorem \ref{thm:lucas_multinomial}, the multinomial coefficient $\binom{r}{\alpha_1,\dots ,\alpha_n}$ is non-zero
    mod $p$ and thus non-zero in $\F_q$ exactly when every one-digit factor
    $\binom{r_t}{\alpha_{1,t},\dots ,\alpha_{n,t}}$ is non-zero.
    Since each $r_t<p$, none of the factorials occurring in that
    one-digit factor is divisible by $p$, so the only condition is the sum constraint $\alpha_{1,t}+\cdots+\alpha_{n,t}=r_t$.

    For fixed $t$, the number of $(\alpha_{1,t},\dots,\alpha_{n,t})$ satisfying that constraint equals the stars-and-bars count $\binom{r_t+n-1}{\,r_t}$. Digit positions are independent, hence the total number of exponent tuples that give non-vanishing coefficients is the product of those counts:
    \[
        \gamma_{n,p}(r)=
        \prod_{t=0}^{s_r}
        \binom{n+r_t-1}{\,r_t},
    \]
    as claimed. 
    
    Naturally, this means that $\P_{(n,m,1), r} \le$ $q^{\gamma_{n,p}(r)}$, since there are $q$ choices for the coefficient of each non-zero term in the multi-nomial expansion. It is not difficult to extend to $k>1$. Consider the $k-$tuple of outputs $(u_1,u_2, \dots, u_k)$: each $u_i$ has at most $q^{\gamma_{n,p}(r)}$ choices. Since the outputs are independent, we get $\P_{(n,m,k), r} \le q^{\gamma_{n,p}(r) \cdot k}$.
\end{proof}

\section{Shallow networks with $r=2$ and $char(\F_q)\neq 2$}
\label{sec:shallow-r=2-p!=2}
    The $i$-th output $f_{i,\ww}$ of the neural network corresponds to a symmetric tensor of order $r$, which we will denote by $A_i$. Then, the output of the neural network $f_{\ww}$ can be identified with a $k$-tuple $(A_1, ..., A_k)^T$ of symmetric tensors, each satisfying the symmetric CP-decomposition
    \begin{equation}
    \label{eq: nmk architecture any r}
        A_i := b_{i1} L_1^{\otimes r} + \cdots + b_{im} L_m^{\otimes r}        
    \end{equation}
    where each $L_j \in \mathbb{F}^n$ is $j$-th row of $W_1$~\citep{LANDSBERG10}.

    However, over finite fields, not every symmetric tensor has a well-defined CP-rank~\citep[Proposition~7.1]{FRIEDLAND13}. Therefore, there exist shallow architectures whose ambient space contains elements that cannot be expressed via the decomposition in~\Cref{eq: nmk architecture any r}. 

    Fortunately, for the square activation ($r=2$) with $char(\F)\neq 2$, every symmetric matrix admits a symmetric rank-one decomposition of the form in~\Cref{eq: nmk architecture any r}. This allows us to compute the cardinality of the neuromanifold for architectures $\dd=(n,m,1),$ and $r=2$ below.

\subsection{Single Output Architectures $\dd=(n,m,1)$ with $r=2,$ $char(\F)\neq 2$}
    For this architecture, the output of the neural network $f_{\ww}$ according to~\Cref{eq: nmk architecture any r} is represented by a single $n \times n$ symmetric matrix of the form
    \begin{equation}
    \label{eq: simultaneous diagonalization formulation}
        A = b_1\, L_1^\top L_1 + \cdots + b_m\, L_m^\top L_m = W_1^TD(W_{2,1})W_1,
    \end{equation}

    where $W_{2,i}$ is the $i$-th row of $W_2$ and $D(\mathbf{v})$ is the diagonal matrix formed by placing the entries of a vector $\mathbf{v}$ on the diagonal. The theorem below (a reformulation of~\citet[Theorem~9.13]{AXLER15}) shows that every symmetric matrix over any field of characteristic $\neq 2$ has a well-defined CP-rank.
    \begin{theorem} [Corollary of {\citet[Theorem~9.13]{AXLER15}}] \label{thm: sym matrices diagonalizable by congruence}
    For every symmetric matrix $C$ over a field $\F$ of characteristic $\neq 2$, there exists a diagonal matrix $D$ and an invertible matrix $Q$ such that $C = Q D Q^\top$.
    \end{theorem}
    This is a direct consequence of the classical linear algebra result that every symmetric bilinear form over a field of characteristic $\neq 2$ can be diagonalized. 
    
    According to~\citet[Lemma~3.3]{KUBJAS24}, the neuromanifold
    \begin{equation}
    \label{eq: neuromanifold description}
        \text{$\mathcal{P}_{(n,m,1),r = 2}(\F)$ is equivalent to the set of symmetric $n\times n$ matrices of rank $\leq m$ }
    \end{equation}
    
    Theorem~\ref{thm: sym matrices diagonalizable by congruence} tells us that the above description can be extended to any field of characteristic $p \neq 2$. Therefore, computing the cardinality of the neuromanifold $\mathcal{P}_{\dd, r}$ for $\dd = (n,m,1)$, $r = 2$ over $\F_q$ is equivalent to counting the number of $n \times n$ symmetric matrices of rank $\leq m$ over $\F_q$.
    
    If $m \geq n$, then clearly $\mathcal{P}_{\dd,r}$ is the entire space of symmetric matrices $S^r(\mathbb{F}^n_q)$:
    \begin{lemma} \label{lem: r=2 p!=2 k=1 m>=n}
        If $\dd = (n,m,1)$, $r = 2$, $p \neq 2$, then
        \begin{align*}
            \mathcal{P}_{\dd, r} = S^r(\mathbb{F}^n_q) \iff m \geq n
        \end{align*}
    \end{lemma}
    For the cases where $m \leq n$, we refer to a paper by~\citet{WILLIAMS69} which counts the number of $n \times n$ symmetric matrices of rank $m$ over $\F_q$.
    \begin{theorem} [{~\citet[Theorem~2]{WILLIAMS69}}] \label{thm: number of symmetric n x n matrices of rank m}
        Let $N_q(n, m)$ denote the number of $n \times n$ symmetric matrices of rank $m$ over a field $\F_q$ of characteristic not equal to 2. Then
        \[
            N_q(n,m) = \prod_{i = 1}^{\lfloor{\frac{m}{2}}\rfloor}\frac{q^{2i}}{q^{2i} - 1} \cdot \prod_{i = 0}^{m - 1} (q^{n - i} - 1).
        \]
    \end{theorem}
        As an immediate corollary,
    \begin{corollary} \label{lem: r=2 p!=2 k=1 m<=n}
        For $m \leq n$,
        \begin{align*}
            |\mathcal{P}_{(n,m,1), 2}(\F_q)| &= \sum_{j = 0}^{m} N_q(n,j)
        \end{align*}
    \end{corollary}
    Next, we focus on architectures with output width $k=2$.
    
\subsection{Shallow Architectures $\dd=(n,n,2)$, $char(\F_q)\neq 2$}
    We now turn to the case of a network with multiple outputs. Here, the structure of the neuromanifold becomes substantially more subtle and is closely related to the problem of simultaneous diagonalization via congruence of symmetric tensors. We focus on the square architecture $\dd=(n,n,2)$ and show that, while the corresponding neuromanifold is Zariski dense over $\mathbb{C}$, its finite-field analogue occupies only a proper fraction of the ambient space.

    \noindent
    \textbf{The case $r=2$.} According to~\Cref{eq: simultaneous diagonalization formulation}, a point $(A_1, A_2)$ belongs to the neuromanifold $\mathcal{P}_{(n,n,2), 2}$ if and only if 
    \begin{equation}
    \label{eq:nn2-r2-decomposition-formula}
        A_i = b_{i1}\, L_1^\top L_1 + \cdots + b_{im}\, L_m^\top L_m = W_1^TD(W_{2,i})W_1
    \end{equation}
    for $i=1,2$. We are going to follow the definition of simultaneous diagonalization of symmetric matrices given in~\citet{BUSTAMANTE20}. 
    \begin{definition}
        A set of $k$ symmetric matrices $A_1,\dots,A_k$ over $\F$ are \textit{simultaneously diagonalizable via congruence} (SDC) if there exists an invertible matrix $Q \in GL_n(\F)$ such that
        \[
            A_i = QD_iQ^T, \quad \text{$D_i$ is diagonal for all $i = 1, ..., k$.}
        \]
    \end{definition}
    It follows from the above definition that if $A_1, A_2$ are simultaneously diagonalizable, then they lie in $\mathcal{P}_{(n,n,2), 2}$. The opposite direction unfortunately does not hold as the matrix $W_1$ may not be invertible. In other words, we can find a pair of symmetric matrices $(A_1,A_2)$ that satisfy decomposition~\Cref{eq:nn2-r2-decomposition-formula}, but are not simultaneously diagonalizable.
    \begin{example}
        If $\dd=(3,3,2)$, then the network output for the forms $L_1(\xx)=x_1, L_2(\xx)=x_2,$ and $L_3(\xx)=x_1+x_2$ equals to
        \[
        \begin{aligned}
            & f_{1,\ww}(x_1,x_2,x_3) = b_{11}x_1^2 + b_{12}x_2^2 + b_{13}(x_1+x_2)^2\\
            & f_{2,\ww}(x_1,x_2,x_3) = b_{21}x_1^2 + b_{22}x_2^2 + b_{23}(x_1+x_2)^2
        \end{aligned}
        \]
        Taking $b_{11}=b_{12}=-1, b_{13}=1$ and $b_{21} = b_{23} =0, b_{22}=1$, we obtain $f_{1,\ww}(\xx) = 2x_1x_2$ and $f_{2,\ww}(\xx)=x_2^2$, which are not simultaneously diagonalizable according to Corollary~\ref{thm: zariski closure (n_n_2) over C filling}. 
    \end{example}
    
    By applying an algorithm for determining SDC given by ~\citet{BUSTAMANTE20}, we can partially extend the results on filling architectures for shallow networks over $\mathbb{C}$ originally given in the paper~\cite{KUBJAS24}. This yields the following.
    
    \begin{theorem} \label{thm: zariski closure (n_n_2) over C filling}
        If $\dd=(n,n,2), r=2,$ and $\F=\mathbb{C}$, then the Zariski closure of $\mathcal{P}_{\dd,r}(\mathbb{C})$ is the entire ambient space $S^2(\mathbb{C}^n)\times S^2(\mathbb{C}^n)$.
    \end{theorem}
    To prove the above, we will apply the following corollary of~\citet[Theorem~7]{BUSTAMANTE20}.
        \begin{corollary} \label{lem: 2-tuple sdc}
            Let $A_1, A_2$ be $n \times n$ symmetric matrices, where $A_1$ is full-rank. Then, $A_1, A_2$ are simultaneously diagonalizable via congruence if and only if there exists a diagonal matrix $D$ and an invertible $n \times n$ matrix $Q$ such that
            $$A_1^{-1}A_2 = QDQ^{-1}.$$
        \end{corollary} 
    
    \begin{proof} [Proof of Theorem \ref{thm: zariski closure (n_n_2) over C filling}]
        We wish to show the \textit{Zariski closure} of $\mathcal{P}_{(n,n,2),r}(\C)$ is the entire ambient space $S_2(\mathbb{C}^n)\times S_2(\mathbb{C}^n)$, for which we shall apply Corollary \ref{lem: 2-tuple sdc}.

        If $A_1$ is full rank and $A_1^{-1}A_2$ has no repeat eigenvalues, then its Jordan normal form is diagonal. That is, there will exist a change of basis matrix $Q$ such that $A_1^{-1}A_2 = QDQ^{-1}$, and $A_1,A_2$ will be SDC by our corollary. 

        $B:= A_1^{-1}A_2 $ has no repeat eigenvalues if and only if the discriminant $\Delta$ of its characteristic polynomial $f(x) = \operatorname{charpoly}(B) = \det(B-Ix)$ is non-zero. The discriminant of a polynomial can be expressed as the \textit{resultant} of the polynomial $f(x)$ and its derivative $f'(x)$. Notably, $f(x)$ and $f'(x)$ are polynomials whose coefficients are polynomials in the entries of $A_1, A_2$, and the resultant is a polynomial in said coefficients. Hence, the resultant is a polynomial in the entries of $A_1, A_2$, and thus 
        $$V_1: \quad \Delta (\operatorname{charpoly}(B)) = 0$$
        is an algebraic variety in the entries of $A_1, A_2$. 

        Meanwhile, the set $ \{\det A_1 = \det A_2 = 0\}$ is also an algebraic variety. Specifically, $\{\det A_1 = \det A_2 = 0\}$ is the variety $$V_2 := \cap_{\substack{i \in \{1,2\} \\ 1 \leq j,k \leq n}} V_{i,j,k}$$ where $$V_{i,j,k} := \{\det (A_i)_{j,k} = 0\}$$ is the variety given by the determinant of the $j,k$-th minor of $A_i$. We then have $\left(S_2(\mathbb{C}^n)\right)^2 \setminus (V_1 \cup V_2)$ as a subset of all 2-tuples $A_1,A_2$ which are SDC. Thus,
        \[
            S_2(\mathbb{C}^n) \setminus (V_1 \cup V_2) \subseteq \mathcal{P}_{(n,n,2),2}.
        \]
        
        $V_1 \cup V_2$ is a proper Zariski-closed subset of $\left(S_2(\mathbb{C}^n)\right)^2 \cong \mathbb{C}^{(n+1)n}$. Hence, its complement $\left(S_2(\mathbb{C}^n)\right)^2 \setminus (V_1 \cup V_2)$ is a non-empty Zariski-open set. Non-empty Zariski-open sets are dense in $\mathbb{C}^k$ for any natural number $k$. Therefore, the Zariski closure of $\left(S_2(\mathbb{C}^n)\right)^2 \setminus (V_1 \cup V_2)$ (and hence of the superset $\mathcal{P}$) is the entire ambient space $\left(S_2(\mathbb{C}^n)\right)^2$, as desired.
    \end{proof}
    One might hope $\mathcal{P}_{(n,n,2),r}(\mathbb{F}_q)$ behaves similarly to its analogue over $\mathbb{C}$. Previously, when $k = 1$, the neuromanifold $\mathcal{P}_{(n,n,1),r}$ fills the ambient space over both $\mathbb{C}$ and $\F_q$. However, their patterns diverge for $k> 1$.

    For example, if $k=2$, the neuromanifold $\mathcal{P}_{(n,n,2),r}(\C)$ contains an open Zariski set whose closure fills the entire ambient space. Contrastingly, we observe the neuromanifold $\mathcal{P}_{(n,n,2),r}(\mathbb{F}_q)$ only fills approximately half of the ambient space.

    \begin{table}[!h]
        \centering
        \begin{tabular}{c|ccc}
             $p$ & $|\mathcal{P}_{(2,2,2), r = 2}(\F_p)|$ & $|S^2(\F_p^2)^2|$ & $\delta_{\dd,r}(p)$ \\ \hline
             3 & 393 & 729 & 0.539...\\
             5 & 7945 & 15625 & 0.509...\\
             7 & 59185 & 117649 & 0.503...\\
             11 & 2415673 & 4826809 & 0.500...
        \end{tabular}
        \caption{Arithmetic Expressivity of $\mathcal{P}_{(2,2,2), r = 2}(\F_p)$}
        \label{eq: 222-r=2}
    \end{table}
    \begin{conjecture} \label{conj: (2_2_2) over F_p}
        The arithmetic expressive capacity of $\mathcal{P}_{(2,2,2),r=2}(\F_p)$ approaches $\frac{1}{2}$ as $p \to \infty$:
        \[
            \lim_{p\to\infty}\delta_{(2,2,2),2}(p)=1/2.
        \]
    \end{conjecture}

    To gain some intuition for Conjecture~\ref{conj: (2_2_2) over F_p}, we compare with the classical picture over $\C$. Over $\mathbb{C}$, two generic (and thus invertible) $2 \times 2$ symmetric matrices $A,B$ are SDC precisely when $A^{-1}B$ is diagonalizable via similarity~\citep{BUSTAMANTE20}. Diagonalizability fails generically if and only if $A^{-1}B$ has discriminant zero (a non-zero discriminant indicates distinct roots of the characteristic polynomial, which implies diagonalizability). This is a Zariski-closed, or probability zero, event. Consequently, over $\C$, the condition for $SDC$ is generally satisfied, and so $\mathcal{P}_{(n,n,2),r}(\C)$ is a Zariski-dense set. 

    \begin{remark}
        This suggests that studying neuromanifolds over $\overline{\F}_q$ is more interesting from a theoretical perspective. However, from a more practical point of view, understanding the expressivity of networks over $\F_q$ helps with understanding the expressivity of networks with quantized weights.  
    \end{remark}
    
    We propose the following two open questions that can advance our understanding of neuromanifolds over finite fields with architectures where $k > 1$. The first concern is finding architectures whose expressive capacity converges to a non-zero fraction.
    \begin{problem} 
        Find all architectures $\dd=(n,m,k)$ over a prime field $\F_p$ whose arithmetic expressive capacity $\delta_{\dd,r}(p)$ converges to a non-zero fraction $0<a/b<1$ as $p \to \infty$.
    \end{problem}
    \textbf{The case $r>2$.}
    The second question is a generalization of the problem of SDC from symmetric matrices to symmetric tensors. The problem of simultaneous diagonalization of symmetric tensors, when viewed through the lens of homogeneous forms, connects naturally to Waring-type problems studied in~\citet{CARLINI2003}. Given homogeneous forms $f_1,\dots,f_k$ of degree $r$, the goal is to find linear forms $L_1,\dots,L_m$ and coefficients $c_{ij}$ such that
    \begin{equation}
    \label{eq: waring sim decomposition}
        f_i = c_{i1} L_1^r + \cdots + c_{im} L_m^r, \quad i = 1,\dots,k.
    \end{equation}
    
    \begin{lemma} 
    \label{lem: sdc tensors lower bound}
        The set of SDC $k$-tuples of $n^{\otimes r}$ tensors is isomorphic to a subset of $\mathcal{P}_{(n,n,k),r}$. Therefore, the number of SDC $k$-tuples is a lower bound for a neuromanifold $|\mathcal{P}_{(n,n,k),r}|$. 
    \end{lemma}
    \begin{proof}
        The first assertion follows directly from~\Cref{eq: waring sim decomposition}. The second assertion is an immediate consequence of the first.
    \end{proof}
    This lemma from choosing $W_1$ to be invertible, and motivates our second question.
    \begin{problem}
    \label{prob: SDC over C or Fq}
        Given a tuple of symmetric tensors $(A_1,A_2,\dots,A_k)$ over $\mathbb{C}$ or over $\F_q$, determine when they are simultaneously diagonalizable via congruence.
    \end{problem}
    
    Problem \ref{prob: SDC over C or Fq} is solved geometrically in the case $\dd = (3,3,2)$ and $r=2$ by~\citet{KUSEJKO15}. In the case $\dd=(n,m,k)$ and $r=2$, Problem \ref{prob: SDC over C or Fq} was solved by~\citet{BUSTAMANTE20} over $\mathbb C$.
    
\section{Shallow architectures $\dd = (n,1,k), (n,2,k)$} 
\label{sec:shallow-m=1-m=2}

    In this section, we study neuromanifolds through the perspective of projective space. One of the main motivations is that the projective point count, as applied in~\citet{BAEZ25}, can be used to understand the ``motivic'' structure of the corresponding complex variety over complex numbers via the Weil conjectures~\cite{BRYAN26}. In other words, studying neuromanifolds over finite fields $\F_q$ helps with understanding the geometry and structure of the neuromanifold and its Zariski closure over $\C$.

    \noindent
    \textbf{Projective space.} Let $\F_q$ be a finite field. Recall that the neuromanifold $\mathcal{P}_{\dd,r}$ is a subset of $\F_q^N$, where $N=\dim(S^r(\F_q^n))^k$ is the dimension of the ambient space over $\F_q$. \Cref{apx: proj space def} gives a brief introduction to a projective space $\P^{N-1}(\F_q)$. We define $\overline{\mathcal{P}_{\dd,r}}$ to be the image of $\mathcal{P}_{\dd,r}$ in the projective space under the map
    \[
        pr:\F_q^{N}\to \P^{N-1}(\F_q),\quad (x_1,\dots,x_N)\mapsto[x_1:\dots:x_N].
    \] 
    The following lemma relates the cardinality of $\overline{\mathcal{P}_{\dd,r}}$ with the non-projective $\mathcal{P}_{\dd,r}$.
    \begin{lemma} \label{lem: count projective vs non-projective}
            $|\mathcal{P}_{\dd,r}| = (q-1)|\overline{\mathcal{P}_{\dd,r}}|+1.$
    \end{lemma}
    \begin{proof}
        The $k$-tuple consisting of all $0$-matrices lies in $\mathcal{P}_{\dd,r}$. For all remaining $k$-tuples $(A_1, \ldots, A_k) \in \mathcal{P}_{\dd,r}$, scaling by any $\lambda \in \F_q \setminus \{0\}$ gives another point $\lambda(A_1, \ldots, A_k) \in \mathcal{P}_{\dd,r}$, distinct if $\lambda \neq 1$. Therefore,
        \[
            |\overline{\mathcal{P}_{\dd,r}}| = \frac{|\mathcal{P}_{\dd,r}| -1}{q-1}.
        \]
        Rearranging gives the desired result.
    \end{proof}
    
    In relation to~\Cref{eq: nmk architecture any r}, the decomposition of the $A_i$'s depends on the scalars $b_{ij}$ and the lines generated by $L_j$. Recall that each $L_j(\xx)$ is defined as the $j$-th row of the vector $W_1\xx$, i.e.,
    \[
        L_j(\xx) = a_{j1}x_1+\dots+a_{jn}x_n.
    \]

    \noindent
    \textbf{The case $m=1,2$}. Let $\mathbb{P}^k \coloneq \mathbb{P}^k(\F_q)$. For $m=1$, we have the following count.
    \begin{lemma} 
    \label{lem: n_1_k}
    For $\dd = (n,1,k)$, any $r$ and any prime $p$,
    \[
        |\overline{\mathcal{P}_{\dd,r}}| = |\overline{\mathcal{P}_{(n,1,1),r}}|\left |\mathbb{P}^{k-1}\right|= \frac{(q^n-1)(q^k-1)}{(q-1)^2}.
    \]
    \end{lemma}
    \begin{proof}
        The elements of $\overline{\mathcal{P}_{\dd,r}}$ are $k$-tuples of symmetric tensors $\left[A_1:...:A_k\right]$ up to scaling, where
        $$A_i = b_i L^{\otimes r}.$$
    
        If we fix $L$ up to scaling, then each $A_i$ lies on the line in $S_r(\F_q^n)$ generated by $L^{\otimes r}$. We claim different choices of $\left[L\right]$ (choices of $L$ up to scaling) partition $\overline{\mathcal{P}_{\dd,r}}$. In other words, we claim given distinct $\left[L_1\right]$, $\left[L_2\right] \in \P(\mathbb{F}_q^n)  \cong \mathbb{P}^{n-1}$ over $\F_q$, the resulting $L_1^{\otimes r}$ and $L_2^{\otimes r}$ define different lines in $S_r(\F_q^n)$ (alternatively stated, $L_1^{\otimes r}$ and $L_2^{\otimes r}$ are not scalar multiples of each other). 
    
        We can assume by change of basis that $L_1 = e_1$ is a standard basis vector. Then, $L_1^{\otimes r}$ is an order-$r$ symmetric tensor with a $1$ at the entry indexed by $(1,\ldots,1)$, and zero everywhere else. If $L_2 = \lambda L_1$ where $\lambda \neq 0$, then the first entry of $L_2$, denoted $(L_2)_1$, must be non-zero. Meanwhile, the $j$-th entry of $L_2$, denoted $(L_2)_j$, must be zero, since there exists an entry corresponding to $(L_2)_1^{r-1}(L_2)_j = 0$ for all $j \neq 1$. Therefore, $L_2$ is a multiple of $e_1 = L_1$. Therefore, if $\left[L_1^{\otimes r}\right] = \left[L_2^{\otimes r}\right]$ in $\P(S_r(\F_q^n))$, then $\left[L_1\right] = \left[L_2\right]$. 
        
        So, choice of $\left[L\right]$ partition $\overline{\mathcal{P}_{\dd,r}}$. We furthermore observe, given a choice of $\left[L\right] \in \P(\F_q^n) \cong \P^{n-1}$, $\left[b_1:...:b_k\right] \in \mathbb{P}^{k-1}$ determines $\left[A_1:...:A_k\right]$ uniquely. Thus,
        $$\overline{\mathcal{P}_{\dd,r}} \cong \P^{n-1} \times \mathbb{P}^{k-1}.$$
    
        By removing the origin and dividing by $q-1$ to account for scalars,
        $$|\P^{i-1}| = \frac{|\F_q^{i}| - 1}{q-1} = \frac{q^i-1}{q-1}$$
        for all $i \in \mathbb{N}$. Therefore,
        $$|\overline{\mathcal{P}_{\dd,r}}| = |\mathbb{P}^{n-1}| |\mathbb{P}^{k-1}| = \frac{(q^n-1)(q^k-1)}{(q-1)^2}.$$
    \end{proof}
    For larger values of $m$, the space $\overline{\mathcal{P}_{\dd,r}}$ can be partitioned by the dimension of the span of the $k$-tuple $(A_1,\dots,A_k)$ as a vector space over $\F_q$. This allows us to describe $\overline{\mathcal{P}_{\dd,r}}$ in terms of spaces associated with CP-rank smaller than $m$, by studying how the rank-1 components combine linearly. For $m = 2$, this yields~\Cref{eq: (n_2_k)-r>1} and~\Cref{eq: (n_2_k)-r = 1}, depending on whether $r>1$ or $r = 1$.
    \begin{proposition}\label{prop: (n_2_k)-r>1}
        For $\dd = (n,2,k)$, $p \nmid r$, $r > 1$,
        \begin{equation} \label{eq: (n_2_k)-r>1}
            \left|\overline{\mathcal{P}_{\dd, r}}\right| = \left|\overline{\mathcal{P}_{(n,2,1), r}}\right| \left|\mathbb{P}^{k-1}\right| + \binom{\left|\overline{\mathcal{P}_{(n,1,1), r}}\right|}{2}\frac{N_{q}(2,k,2)}{q-1}
        \end{equation}
        where $N_q(n,k,m)$ represents the number of $k\times n$ matrices of rank $m$ over $\F_q$ (see~\Cref{eq: k x n matrices of rank m}).
    \end{proposition} 
    \begin{proof}
        The elements of $\overline{\mathcal{P}_{\dd, r}}$ are projective equivalence classes of ordered $k$-tuples of $n^{\otimes r}$ symmetric tensors $\left[A_1:...:A_k\right]$ of the form 
        $$A_i = b_{i1} L_1^{\otimes r} + b_{i2} L_2^{\otimes r}$$
        where $\left[L_1^{\otimes r}\right], \left[L_2^{\otimes r}\right] \in \overline{\mathcal{P}_{(n,1,1),r}}$ are equivalence classes of rank-1 $n^{\otimes r}$ symmetric tensors. 
    
        We consider $L_1^{\otimes r}$, $L_2^{\otimes r}$ as elements of a vector space over $\mathbb{F}_q$. Then, since we have two basis elements $L_1^{\otimes r}$ and $L_2^{\otimes r}$, $\left[A_1:...:A_k\right]$ spans 1 or 2 dimensions. 
    
        If the tensors in $\left[A_1:...:A_k\right]$ span 1 dimension, then, in fact, all $A_i$ are of the form
        $A_i = b_iA$
        for some $A \in \mathcal{P}_{(n,2,1),r}$. The choice of $\left[A\right]$ and parameterization $\left[B\right] = \left[b_1, b_2, ..., b_k\right] \in \mathbb{P}^{k-1}$ uniquely determines $\left[A_1:...:A_k\right]$. Hence, 
        $$\{\text{1-dimensional} \left[A_1:...:A_k\right]\} \cong \overline{\mathcal{P}_{(n,2,1), r}} \times \mathbb{P}^{k-1}$$ 
        and we obtain the term
        $$\left|\overline{\mathcal{P}_{(n,2,1), r}}\right| \left|\mathbb{P}^{k-1}\right|.$$
    
        In the case of a 2-dimensional span by $\left[A_1:...:A_k\right]$, we begin by arguing the choice of basis pair $\left[L_1^{\otimes r}\right],\left[L_2^{\otimes r}\right]$ partitions the set of 2-dimensional $\left[A_1:...:A_k\right]$ disjointly.
        
        For the sake of contradiction, suppose some 2-dimensional $\left[A_1:...:A_k\right]$ may be constructed by two different basis pairs $\left[L_1^{\otimes r}\right],\left[L_2^{\otimes r}\right]$ and $\left[L_3^{\otimes r}\right],\left[L_4^{\otimes r}\right]$. Since $\left[A_1:...:A_k\right]$ is 2-dimensional, the $A_i$ form an alternative spanning set for the space spanned by $L_1^{\otimes r}$ and $L_2^{\otimes r}$ (or $ L_3^{\otimes r}, L_4^{\otimes r}$) and we can therefore recover the original basis $L_1^{\otimes r},L_2^{\otimes r}$ and $ L_3^{\otimes r}, L_4^{\otimes r}$ as a linear combination of the $A_i$'s. Consequently, $L_3^{\otimes r}, L_4^{\otimes r}$ are both in the span of $L_1^{\otimes r},L_2^{\otimes r}$, and the two pairs form the same plane in $S_r(\F_q^n)$. Note that for the pairs to be distinct, at least one of these linear combinations of $L_1^{\otimes r},L_2^{\otimes r}$ is non-trivial. More precisely, there exists $c_1,c_2 \neq 0$ such that
        $$L_j^{\otimes r} = c_1L_1^{\otimes r} + c_2L_2^{\otimes r}$$
        where $j = 3$ or $4$.

        Because $\left[L_1^{\otimes r}\right] \neq \left[L_2^{\otimes r}\right]$, we have $L_1^{\otimes r}$, $L_2^{\otimes r}$ are linearly independent, and consequently so are $L_1$ and $L_2$. If we apply a change of basis such that $L_1 = (1, 0, 0, ..., 0)$ and $L_2 = (0, 1, 0, ..., 0)$ are standard basis vectors, then $c_1L_1^{\otimes r} + c_2L_2^{\otimes r}$ is an $n^{\otimes r}$ tensor with one $c_1$ and one $c_2$ along the diagonal and 0's everywhere else. Since $r > 1$, therefore this tensor is of symmetric CP-rank 2. In other words, there does not exist an $L_3$ such that $L_3^{\otimes r} = c_1L_1^{\otimes r} + c_2L_2^{\otimes r}$ where $c_1, c_2 \neq 0$. Therefore, we have a contradiction, and in fact, our choice of distinct $\left[L_1^{\otimes r}\right],\left[L_2^{\otimes r}\right]$ will partition the set of 2-dimensional $\left[A_1:...:A_k\right]$'s disjointly.
        
        Fixing a choice of distinct $\left[L_1^{\otimes r}\right]$ and $\left[L_2^{\otimes r}\right] \in \overline{\mathcal{P}_{(n,1,1),r}}$, $\left[A_1:...:A_k\right]$ is 2-dimensional if and only if the coefficients as vectors $\left(\begin{array}{c} b_{i1} \\ b_{i2}\end{array}\right)$ span two dimensions. That is, each possible $\left[A_1:...:A_k\right]$ is uniquely parameterized by a choice of basis pair $\left[L_1^{\otimes r}\right], \left[L_2^{\otimes r}\right] \in \overline{\mathcal{P}_{(n,1,1),r}}$ and some rank-2 matrix
        \[
        \left[(B)\right] = \left[\left(\begin{array}{cccc}
            b_{11} & b_{21} & \hdots & b_{k1} \\
            b_{12} & b_{22} & \hdots & b_{k2} 
        \end{array}\right)^T\right] \in \mathbb{P}(M^{2 \times k}).
        \]
        There are $\binom{\left|\overline{\mathcal{P}_{(n,1,1), r}}\right|}{2}$ ways to choose $\left[L_1^{\otimes r}\right],\left[L_2^{\otimes r}\right]$. Fixing $\left[L_1^{\otimes r}\right],\left[L_2^{\otimes r}\right]$, the number of choices of matrix projective point $\left[(B)\right]$ is $\frac{N_{q}(2,k,2)}{q - 1} = \frac{(q^k - 1)(q^k - q)}{q-1}$. Hence, we obtain our second term, the count of 2-dimensional $\left[A_1:...:A_k\right]$,
        \[
            \binom{\left|\overline{\mathcal{P}_{(n,1,1), r}}\right|}{2} \frac{N_{q}(2,k,2)}{q - 1}.
        \]
        Therefore,
        \begin{align*}
            \left|\overline{\mathcal{P}_{\dd, r}}\right| 
            &= \underbrace{\left|\overline{\mathcal{P}_{(n,2,1),r}}\right| \cdot \left|\mathbb{P}^{k-1}\right|}_{\text{$1$-dimensional span}}
            + \underbrace{\binom{\left|\overline{\mathcal{P}_{(n,1,1),r}}\right|}{2} \cdot \frac{N_q(2,k,2)}{q-1}}_{\text{$2$-dimensional span}}.
        \end{align*}
        as desired.
    \end{proof}

\begin{proposition}\label{prop: (n_2_k)-r = 1}
    For $\dd = (n,2,k)$, $r = 1$,
    \begin{equation} \label{eq: (n_2_k)-r = 1}
        \left|\overline{\mathcal{P}_{\dd, r}}\right| = \left|\overline{\mathcal{P}_{(n,2,1), r}}\right| \left|\mathbb{P}^{k-1}\right| + \frac{2}{(q+1)q}\binom{\left|\overline{\mathcal{P}_{(n,1,1), r}}\right|}{2}\frac{N_{q}(2,k,2)}{q-1}
    \end{equation}
\end{proposition}
\begin{proof}
    Compared to the previous proposition, the term corresponding to $k$-tuples $\left[A_1: \ldots: A_k\right]$ with a 1-dimensional span remains unchanged. 
    
    However, the contribution from $k$-tuples $\left[A_1: \ldots: A_k\right]$ with 2-dimensional span is divided by $\frac{q(q+1)}{2}$. This factor arises because, when $r = 1$, linear combinations of $L_1^{\otimes r}, L_2^{\otimes r}$ are no longer rank 2, but rank 1. More specifically, we have
    $$c_1 L_1^{\otimes r}+ c_2L_2^{\otimes r} = c_1 L_1 + c_2L_2 = L_3 = L_3^{\otimes r}.$$
    so any linear combination is itself another vector $L_3 \in \F_q^n$. 
    
    Consequently, the choice of basis pair $\left[L_1\right], \left[L_2\right]$ corresponding to the plane spanned by $\left[A_1: \ldots: A_k\right]$ is no longer unique, unlike in the case $r > 1$ considered in Proposition \ref{prop: (n_2_k)-r>1}. For instance, one could take $L_3 = L_1 + L_2$ and $L_1$ as the spanning pair instead. In fact, any two linearly independent elements of the plane will suffice as a basis, since with $r =1$, each tensor in the plane is a vector.
    
    The number of unordered pairs of linearly independent elements in a 2-dimensional subspace over $\F_q$ is 
    \[
    \frac{(q^2 - 1)(q^2 - q)}{2}.
    \]
    After projectivising (i.e., accounting for scaling by a non-zero scalar), we have
    \[
    \frac{(q^2 - 1)(q^2 - q)}{2} \cdot \frac{1}{(q-1)^2} = \frac{q(q+1)}{2}.
    \]
    Hence, each $2$-dimensional $\left[A_1: \ldots: A_k\right]$ can be attained by $\frac{q(q+1)}{2}$ distinct basis pairs $\left[L_1^{\otimes r}\right], \left[L_2^{\otimes r}\right]$. Therefore, we divide by this factor to adjust for overcounting.
\end{proof}

\begin{remark}
    The proof of Proposition~\ref{prop: (n_2_k)-r = 1} relies on the fact that no nontrivial linear combination
    \[
    c_1 L_1^{\otimes r} + c_2 L_2^{\otimes r} \qquad (c_1,c_2 \neq 0)
    \]
    is a symmetric rank–1 tensor when \(r>1\). For any \(k\)-tuple \(\left[A_1:\cdots:A_k\right]\) spanning a \(2\)-dimensional space, its generating pair
    \(
    \left[L_1^{\otimes r}\right], \left[L_2^{\otimes r}\right]
    \)
    is thus uniquely determined.
    
    This yields a disjoint parametrization by basis pairs, making combinatorial 
    counting in the \(m=2\) case tractable.
    
    For \(m \ge 3\), this uniqueness fails.  
    A \(3\)-dimensional span of symmetric rank–1 tensors may admit multiple distinct 
    decompositions into triples 
    \(\{L_1^{\otimes r}, L_2^{\otimes r}, L_3^{\otimes r}\}\). 
    As a result, the clean partition used for \(m=2\) no longer exists and, consequently, the combinatorics become substantially more involved. For this reason, we currently do not have analogous formulae for larger values of $m$.
\end{remark}
The next section will address the case $p \mid r$. In particular, Proposition \ref{prop:architecture-M_dd_ip_eq_M_dd_i} will state
\begin{equation} \label{eq: reduction}
    |\mathcal{P}_{\dd, ip}| = |\mathcal{P}_{\dd, i}|,
\end{equation}
which allows us to reduce architectures with $p \mid r$ to the situation $p \nmid r$ considered in the previous two propositions. As a first application of this identity, we obtain the following corollary that generalizes the prior two propositions to the cases where $p$ may divide $r$.
\begin{corollary} \label{lem: dd_n_2_k}
    For $\dd = (n,2,k)$, if $r \neq p^i$, $\left|\overline{\mathcal{P}_{\dd, r}}\right|$ satisfies~\Cref{eq: (n_2_k)-r>1}. If $r = p^i$, then $\left|\overline{\mathcal{P}_{\dd, r}}\right|$ satisfies~\Cref{eq: (n_2_k)-r = 1}.
\end{corollary}

\section{Shallow architectures with $char(\F_q)\,|\, r$}
\label{sec:shallow-p|r}
In this section, we consider architectures where the activation degree $r$ of the activation function $\sigma$ is divisible by the prime characteristic of a finite field $F$. Our first proposition equates the cardinality of the neuromanifold with $r=ip$ for some $i\in\mathbb{N}$ to the cardinality of the neuromanifold with $r=i$. In other words, we have the following result. 
\begin{proposition} 
\label{prop:architecture-M_dd_ip_eq_M_dd_i}
    If $\dd=(n,m,k),\ char(\F) = p$, $\sigma(x) = x^r$ where $r = ip$ for $i\in\mathbb{N}$, then
    \[
        |{\mathcal{P}}_{\dd,r}|=|{\mathcal{P}}_{\dd,i}|.
    \]
\end{proposition} 
\begin{proof} 
    As we defined above, let $a_{ij}$ be the $(i,j)$th entry of $W_1$. Direct computation shows that $k$th entry of $\sigma_1 W_1(\xx)$ is equal to 
    \[
        (\sigma_1 W_1(\xx))_k\coloneqq \left ( \sum_{i=1}^n a_{ki}x_i \right )^{ip}
    \]
    Since $p$ is prime, then we have the following identity
    \[
        (a+b)^p \equiv a^p+b^p\text{ for all }a,b\in\F.
    \]
    Computing the $k$th term, for example, gives that
    \[
        \left (\sum_{i=1}^n a_{ki}x_i\right)^p = \left (a_{k1}x_1 + \sum_{i=2}^n a_{ki}x_i \right)^p = (a_{k1}x_1)^p + \left (\sum_{i=2}^n a_{ki}x_i\right)^p
    \]
    If we now recursively simplify the last term, we will eventually get the following:
    \[
        \left (\sum_{i=1}^n a_{ki}x_i\right)^p = \sum_{i=1}^n a_{ki}^px_i^p
    \]
    Thus, 
    \[
        \left (\sum_{i=1}^n a_{ki}x_i  \right)^{ip} = \left (\sum_{i=1}^n a_{ki}^px_i^p\right )^i 
    \]
    The same computation done with $r = i$ will give us the same expression but with $a_{ki}x_i$ replacing $a_{ki}^px_i^p$. It is immediate that the number of polynomials in $\mathcal{P}_{\dd,ip}$ is at most the number of polynomials in $\mathcal{P}_{\dd,i}$ since $a_{ki}^p \in \F$ for all any $a_{ki} \in \F$. Thus, to show that $|\mathcal{P}_{\dd,ip}| =   |\mathcal{P}_{\dd,i}|$, it suffices to prove that every element of $\F$ is a $p$th power, to allow the coefficient before the $x_i^{p}$ term to be any element from $\F$. Fortunately, it is a standard result in algebra that all finite fields are perfect~\citep[Theorem~3]{CONRAD1995}, which means that for all $b \in \F$, there exists $a\in \F$ such that $b = a^p$, if $char(\F) = p$. Now consider arbitrary $b_{ki} \in \F$; we can choose $a_{ki} \in \F$ such that $b_{ki} = a_{ki}^p$. Let $y_i$ denote $x_i^p$. Then we have
     \[
     (\sigma_1 W_1(\xx))_k = \left (\sum_{i=1}^n b_{ki}y_i\right )^i 
     \]
     Now we observe that this is exactly the same as $(\sigma_1 W_1(\xx))_k$ when computed with $r = i$ instead. Since all polynomials are in the form of $W_2(\sigma_1(W_1(\xx)))$, and $W_2$ does not depend on $r$, we can conclude that $|\mathcal{P}_{\dd,ip}| =   |\mathcal{P}_{\dd,i}|$.
\end{proof}

\begin{corollary}
\label{cor: architecture - get rid of ps in r prime decomposition}
    If $\dd=(n,m,k),\ char(\F) = p$, $\sigma(x) = x^r$ where $r = ip^\ell$ for $\ell,i\in\mathbb{N}$ and $gcd(p,i)= 1$, then 
    \[
        |{\mathcal{P}}_{\dd,r}|=|{\mathcal{P}}_{\dd,i}|.
    \]
\end{corollary}
\begin{proof}
    Apply Proposition~\ref{prop:architecture-M_dd_ip_eq_M_dd_i} inductively to $|{\mathcal{P}}_{\dd,r}|$.
\end{proof}
The case $s=1$, i.e., $r = p^\ell$ for some positive integer $\ell$, can be computed explicitly as shown below.
\begin{proposition}\label{prop: r = p^i}
    If $\dd = (n,m,k)$, $\operatorname{char}(\F) = p$, and $\sigma(x) = x^r$ with $r = p^\ell$ for some positive integer $\ell$, then $|\mathcal{P}_{\dd,r}|$ is precisely the number of $k \times n$ matrices over $\F$ of rank at most $m$.
\end{proposition} 

Before we begin the proof, we state the following useful lemma:
\begin{lemma}[{\citet[Theorem~3.56]{AXLER15}}] \label{lem: matrix decomposition rank <= m}
    A matrix $C \in \F^{k \times n}$ is representable as a product of two matrices $B \times A$, where $B \in \F^{k \times m}$, $A \in \F^{m \times n}$ if and only if $C$ has rank $\leq m$.
\end{lemma}

\begin{proof}[Proof of Proposition \ref{prop: r = p^i}]

By Corollary \ref{cor: architecture - get rid of ps in r prime decomposition}, we have that for all $i \in \mathbb{N}$, $|{\mathcal{P}}_{\dd,p^i}| = |{\mathcal{P}}_{\dd,1}|$. Now we need to compute the cardinality of $\mathcal{P}_{\dd,1}$, where the network output is given by 
\[
    f_{\ww}(\xx)=W_2\sigma_1 W_1\xx = W_2W_1\xx. 
\]
By direct computations, we see that in order to learn $f(\xx)=C\xx$, we must have the following decomposition: 
\[
    C = W_2\times W_1.
\]
where $C \in \F^{k \times n}, W_2 \in \F^{k \times m}$, and $W_1 \in \F^{m \times n}$. Thus, according to Lemma \ref{lem: matrix decomposition rank <= m}, the neuromanifold exactly corresponds to the set of matrices with rank at most $m$. 
\end{proof}

\begin{corollary} 
\label{cor: char(F) = p, m >= n or k}
    If $char(\F_q)=p$, $r = p^\ell$, and $N_q(n,k,j)$ is defined as in~\Cref{eq: k x n matrices of rank m} to be the number of  the number of $k\times n$ matrices of rank $m$ over a finite field $\F_q$, then $$|\mathcal{P}_{\dd,r}| = \min\left(q^{k \times n}, \sum_{j = 1}^m N_q(n,k,j) + 1\right),$$
\end{corollary} 
\begin{proof}
    By Proposition \ref{prop: r = p^i}, we find that $|\mathcal{P}_{\dd,1}|$ is the number of all $k \times n$ matrices with rank $\leq m$.
    
    If $m \geq \min(n,k)$, then since all matrices of size $k \times n$ have rank at most $\min(n,k)$ and $m \geq \min(n,k)$, all $k \times n$ matrices satisfy the condition of having rank at most $m$. Thus, by Proposition \ref{prop: r = p^i}, $|\mathcal{P}_{\dd,r}| = q^{k \times n}$.
    
    By~\Cref{eq: k x n matrices of rank m}, we have that $N_q(n,k,j)$ is the number of $k \times n$ matrices with rank exactly $j$. Thus, if $m < \min(n,k)$, $|\mathcal{P}_{\dd,r}| =  \sum_{j=1}^m N_{q}(n,k,j) + 1$, where $1$ corresponds to the zero matrix, which has rank $0$. And this expression is indeed smaller than $p^{k \times n}$ since the latter expression also includes matrices with rank $ > m $ but $ \leq \min(n,k)$
\end{proof}

\section{Conclusions} 
\label{sec:conclusion}

    \textbf{Summary.} In this work, we developed an algebraic framework for analyzing shallow PNNs over finite fields. By formulating algebraic expressivity in terms of the cardinality of neuromanifolds, we established general upper bounds, lower bounds, and explicit formulae for a range of architectures. We also highlighted differences between neuromanifolds defined over finite fields and their counterparts over $\mathbb{C}$, particularly in cases where finite field structures constrain the filling property.

    \textbf{Future work.} A natural extension of our work is to analyze deeper architectures (with multiple hidden layers) to investigate how the composition of several layers affects the size and geometry of the neuromanifolds. Another direction includes the study of more general activations. For example, in~\citet{SHAPIRA2023} shows that deep ReLU networks are more efficient than shallow ones at expressing homogeneous polynomials, with fewer units needed for the $d$-variable product function as depth increases. By analogy, it would be interesting in future work to study whether deeper PNNs similarly outperform shallow ones over finite fields.

    The projective cardinalities introduced and computed in Section~\ref{sec:shallow-m=1-m=2} can also be combined with the extension of the scalars from $\F_q$ to $\overline{\F}_q$ and the Weil conjectures to get new insights into their topological structure, the Zariski closure of neuromanifolds over $\C$.
    
    In addition, we identified several open problems throughout our main sections that merit further examination, such as the question of simultaneous diagonalization of tensors over finite fields.

    \textbf{Conclusion.} Our results provide evidence that recent algebraic approaches can play an important role in understanding the expressivity of neural networks, complementing the more familiar statistical and probabilistic perspectives.

\printcredits
\bibliographystyle{cas-model2-names}
\bibliography{FFT_2026_spring/cas-refs-main} 

\appendix
\section{Appendix}

\subsection{Projective space definition}
\label{apx: proj space def}

An affine space $\F^{n}$ is the $n$-fold Cartesian product of $\F$. A projective space $\P^{n-1}(\F)$ is the set 
    \[
        \P^{n-1}(\F) := (\F^{n}-\{(0,\dots,0)\})/\sim 
    \]
    where $(x_1,\dots,x_n)\sim(y_1,\dots,y_n)$ if and only if there exists a non-zero scalar $\lambda\in\F$ such that $x_i=\lambda y_i$ for all $i$. Intuitively, projective space treats two vectors as the same if they lie on the same line through the origin. Thus, it can be viewed as the set of all lines through the origin in $\F^{n}$. A point in the projective space $\P^{n-1}(\F)$ will be indicated by $[x_1:\dots:x_n]$ where
    \[
        [x_1 : \cdots : x_n]\coloneq\left\{(\lambda x_1 , \cdots , \lambda x_n) : \lambda \in (\F-\{0\})\right\}.   
    \]

\subsection{Multinomial Congruence}

\begin{theorem}[Multinomial Congruence,~\citep{CONRAD1995}]
\label{thm:lucas_multinomial}
    Let $p$ be a prime. Fix integers $d\ge 0$ and $t\ge 1$. Suppose $s_0,s_1,\dots ,s_t\ge 0$. Write
    \[
        s_0=c_0+c_1p+\dots +c_dp^{d},\qquad
        0\le c_i\le p-1\ (i<d),
    \]
    \[
        s_j=c_{0j}+c_{1j}p+\dots +c_{dj}p^{d},\qquad
        0\le c_{ij}\le p-1\ (i<d,\;1\le j\le t),
    \]
    with $c_d,c_{dj}\ge 0$. Then
    \[
        \binom{s_0}{s_1,\dots ,s_t}\equiv
        \binom{c_0}{c_{01},\dots ,c_{0t}}\;
        \binom{c_1}{c_{11},\dots ,c_{1t}}\;\cdots\;
        \binom{c_d}{c_{d1},\dots ,c_{dt}}
        \pmod p.
    \]
\end{theorem}

\begin{example}
    If $\dd=(2,1,1), r=2$, and $p=2$, then the output of the neural network is equal to
    \[
        f_{\ww}(\xx)=b_{11}(a_{11}x_1+a_{12}x_2)^2 =b_{11}a_{11}^2x_1^2+2b_{11}a_{11}a_{12}x_1x_2+ b_{11}a_{12}^2x_2^2 = b_{11}a_{11}^2x_1^2+ b_{11}a_{12}^2x_2^2.
    \]
    Observe that $r = 0\cdot2^0+1\cdot2^1$, so $r_0=0$, $r_1 = 1$, and $s_r=1$. Then we have
    \[
        \gamma_{2,2}(2)={r_0+2-1 \choose r_0}{r_1+2-1 \choose r_1}={1 \choose 0}{2 \choose 1}=2.
    \]
    This exactly gives us $|\mathcal{P}_{(2,1,1),2}|\leq 2^2$. 
\end{example}

\end{document}